\documentclass[11pt, reqno]{amsart}
\usepackage[utf8]{inputenc}
\usepackage{enumerate, xcolor, graphicx}
\definecolor{Magenta}{cmyk}{0,1,0,0}
\definecolor{dgreen}{rgb}{0.,0.6,0.}

\usepackage[top=27mm,bottom=27mm,left=25mm,right=25mm]{geometry}

\usepackage{amssymb,amsthm,amsmath,amsfonts}
\usepackage{diagbox}

\usepackage{mathrsfs, mathtools, mathdots}
\usepackage{paralist,fancyhdr,etoolbox,nicefrac,mathpazo}
\usepackage{parskip,xfrac}
\usepackage{datetime}
\usepackage{hyperref}
\usepackage{cleveref}

\newtheorem{theorem}{Theorem}[]
\newtheorem{definition}{Definition}

\newtheorem{lemma}{Lemma}
\newtheorem{proposition}{Proposition}

\theoremstyle{remark}
\newtheorem{remark}{Remark}

\newcommand*{\mvcenter}[1]{\vcenter{\hbox{$\displaystyle #1$}}}
\newcommand{\Zdisc}{disc_{\mathbb{Z}}}
\newcommand{\disc}{disc}
\newcommand{\ncp}[1]{\langle#1\rangle}
\newcommand\tab[1][1cm]{\hspace*{#1}}
\newcommand{\Q}{\mathbb{Q}}
\newcommand{\Qbar}{\overline{\mathbb{Q}}}
\newcommand{\Z}{\mathbb{Z}}

\newcommand{\Zbar}{\overline{\mathbb{Z}}}
\newcommand{\R}{\mathbb{R}}
\newcommand{\C}{\mathbb{C}}
\newcommand{\K}{\mathbb{K}}

\newcommand{\m}{\mathfrak{m}}

\renewcommand{\O}{\mathcal{O}}

\newcommand{\F}{\mathbb{F}}

\begin{document}

\title[Weakly Divisible Rings]{Weakly Divisible Rings}

\author[G.D.Patil]{Gaurav Digambar Patil}
\address{Department of Mathematics, University of Toronto, Bahen Centre, 40 St. George Street, Room 6290, Toronto, Ontario, Canada, M5S 2E4}
\email{g.patil@mail.utoronto.ca}

\date{\today}

\begin{abstract}
   We define a new class of rings parameterized by binary forms of a certain type, and give an effective lower bound for the number of such rings whose discriminant is less than a bound $X$. We also obtain a lower bound for the number of number fields whose ring of integers lies in the above class and whose discriminant is less than a bound $X$. Our results improve an estimate of Bhargava-Shankar-Wang in \cite{bhargava2022squarefree}. In particular we show the following:
   \begin{itemize}
   \item When $n\ge 4,$ the number of rings of rank $n$ over $\Z$ with discriminant less than or equal to $X$ is 
   $$
   \gg_n X^{\frac{1}{2}+\frac{1}{n-\frac{4}{3}}}.
   $$
   \item When $n\ge 6,$ the number of number fields of degree $n$ with discriminant less than $X$ is
   $$
   \gg_{n,\epsilon} X^{\frac{1}{2} +\frac{1}{n-1} +  \frac{(n-3)r_n}{(n-2)(n-1)}-\epsilon}
   $$ 
   where $r_n=\frac{\eta_n}{n^2-4n+3-2\eta_n (n+\frac{2}{n-2})}$ and where $\eta_n$ is $\frac{1}{5n}$ if $n$ is odd and is $\frac{1}{88n^6}$ when $n$ is even.
   \end{itemize}
\end{abstract}
\maketitle

\section{\bf Introduction}

\bigskip

\begin{definition}
    We define
    $$
    N(X:n):=\#\{K:[K:\Q]=n, Gal(\bar{K}/\Q)\simeq S_n,\disc(K)<X\}.
    $$
\end{definition}
Malle's Conjecture in \cite{MalleConjII} for the group $S_n$ tells us that
$$
N(X:n)\simeq c_n X.
$$
where $c_n$ is a constant dependent on $n.$

Previously, Malle's Conjecture was known for the case $n=3,$ due to the work of Davenport and Heilbronn building off the parametrization of cubic rings developed by Levi-Delone-Faddeev-Gan-Gross-Savin. Until recently there was very little progress on Malle's conjecture. Bhargava showed Malle's conjecture for $n=4$ and $n=5$, using his parametrizations for quartic and quintic rings.

In general, it is often difficult to give any lower or upper bound for $N(X:n)$ of expected order. There has been some recent progress towards the upper bounds for $N(X:n).$  

We will focus on the lower bound in this paper. 
In \cite{BSW-monicsquarefree}, Bhargava-Shankar-Wang show that 
$$
N(X,n)\gg_n X^{\frac{1}{2}+\frac{1}{n}}
$$
and in \cite{bhargava2022squarefree} they show 
$$
N(X,n)\gg_n X^{\frac{1}{2}+\frac{1}{n-1}}.
$$
The first estimate, counts the number of number fields whose ring of integers is a monogenic ring. A monogenic ring is a ring that can be identified as $\Z[\alpha]$ where $\alpha\in \Zbar.$ Hence, a monogenic ring has a monic polynomial canonically associated to it. This allows us to count monic polynomials instead of counting number fields.

Similarly, the second result counts the number of number fields whose ring of integers is a binary ring. A binary ring is a ring of the form $\Z[\delta]\cap\Z[\delta^{-1}]$ where $\delta\in \Qbar.$ Hence, one can naturally associate a binary form of degree equal to the degree of the number field to a binary ring. This again allows us replace the problem of counting number fields with the problem of counting binary forms.

It is expected that the number of number fields whose ring of integers in monogenic and have a discriminant bounded by $X$ is $O(X^{\frac{1}{2}+\frac{1}{n}}).$ Similarly, we expect that the number of number fields whose ring of integers is a binary ring and whose discriminant does not exceed $X$ is $O(X^{\frac{1}{2}+\frac{1}{n-1}}).$ Thus, finding more general parametrizations for rings and therefore rings of integers of number fields is crucial to obtaining better lower bounds on $N(X:n)$.

This is further evident in the cases where Malle's Conjecture for non-abelian Galois groups from \cite{MalleConjII} is completely settled. Bhargava, for example, proves Malle's Conjecture for $n=4$ and $n=5$ using interesting parametrizations for quartic and quintic rings. In this paper, we provide a parametrization that works to improve the power of $X$ in the lower bound of Bhargava-Shankar-Wang.

The second problem that comes up in getting a lower bound is injectivity in the parametrization that we are discussing. At times the parametrization is sufficiently restrictive that one does not have to specifically worry about this. Cases where this happens include 
\begin{itemize}
    \item $n=2$ case, where the problem reduced to counting monic polynomials where the second lead co-efficient is $0$ or $1$ and where the polynomial has a square free discriminant.
    \item $n=3$ case, where the Levi-Delone-Faddeev-Gan-Gross-Savin classification of cubic rings gives tells us that two binary cubic forms will be associated to the same cubic ring if and only if the forms are in the same orbit of the canonical action of $GL_2(\Z)\times GL_1(\Z)$ on the space of binary cubic forms.
    \item $n=4$ and $n=5$, where again Bhargava's classification laws give a clear arithmetic answer to when the rings associated to corresponding arithmetic objects are the same.
\end{itemize}
This problem has been handled in a more effective sense in \cite{BSW-monicsquarefree} and in \cite{bhargava2022squarefree}, where they find a subspace where injectivity holds.

The third and the final problem is showing an arithmetic Bertini type theorem in the context of the parametrization. This allows us to sieve out objects that do not correspond to the full ring of integers. This occupies most of the paper in \cite{BSW-monicsquarefree} and \cite{bhargava2022squarefree}. 
\medskip\par
In this paper, our goal is to replace monogenic and binary rings with another class, namely
a class that we call {\em weakly divisible rings} (defined in \cref{Defining}). The corresponding polynomials were 
actually considered in the work of Bhargava, Shankar and Wang where they appear in error-term estimates. 
In this paper, we consider the binary ring associated to such polynomials and then construct an explicit
larger ring which lies between this binary ring and the full ring of integers of the associated
quotient field. We then count the number of such rings.
\medskip\par
In section 2, we state the main theorems. 
In section 3, we start by reviewing the properties of a binary ring and then define ``weakly divisible" polynomials and ``weakly divisible rings" associated to said polynomials. This gives a more general parametrization than binary rings which in turn more general than monogenic rings. 
In section 4, we discuss injectivity of the association of ``weakly divisible ring" to appropriate weakly divisible polynomial. We use techniques in Minkowski reduction similar to those in \cite{BSW-monicsquarefree} and adapt them to our purposes.
In section 5, we count the number of rings with a bounded discriminant to show \cref{1}.
In section 6, we use the tail end estimates in \cite{bhargava2022squarefree} and sieve the weakly divisible rings to get number fields with a weakly divisible ring of integers showing \cref{2}. This sieve can be easily better adapted to sieving weakly divisible rings, which we will do in some upcoming papers.

\section{\bf Main theorems}

\bigskip

We show the following theorems:
\begin{theorem}[Ring Count]\label{1}
    When $n\ge 4,$ the number of isomorphism classes of integral domains of rank $n$ over $\Z,$  with discriminant less than or equal to $X$ is 
    $$
    \gg_n X^{\frac{1}{2}+\frac{1}{n-\frac{4}{3}}}
    $$ 
\end{theorem}
The previously best available bound here was $\gg_n X^{\frac{1}{2} +\frac{1}{n-1}}.$ It isn't mentioned anywhere explicitly, but we suspect it was in the original Nakagawa Paper \cite{Nakagawa1989}.

We use the power saving nature of the tail estimate theorems in the Bhargava-Shankar-Wang's work in \cite{bhargava2022squarefree} to show that
\begin{theorem}[Number-field count]\label{2}
    When $n\ge 6,$ the number of isomorphism classes of degree $n$-number fields with Galois Group $S_n$ with discriminant less than or equal to $X$ is
    $$
    \gg_{n,\epsilon} X^{\frac{1}{2} +\frac{1}{n-1} +  \frac{(n-3)r_n}{(n-2)(n-1)}-\epsilon}
   $$ 
   where $r_n=\frac{\eta_n}{n^2-4n+3-2\eta_n (n+\frac{2}{n-2})}$ and where $\eta_n$ is $\frac{1}{5n}$ if $n$ is odd and is $\frac{1}{88n^6}$ when $n$ is even.
\end{theorem}

This is an improvement on the lower bound for the same quantity in \cite{bhargava2022squarefree}. In that paper the lower bound is 
$$
\gg_n X^{\frac{1}{2}+\frac{1}{n-1}}.
$$
For example, when $n=6$ this lower bound amounts to 
$$
N(X,6)\gg X^{0.7}
$$
where as our bound gives, 
$$
N(X,6)\gg X^{0.70000000243\cdots}.
$$
For $n=7$, the previous lower bound amounts to 
$$
N(X,7)\gg X^{0.66666666\cdots}
$$
where as our bound gives, 
$$
N(X,7)\gg X^{0.66682824365\cdots}.
$$
We believe this is the best lower bound in the literature, at this moment. Our result is not optimal as the tail end estimates used can be better adapted to this specific sieving problem. In a future paper we will adapt the sieve appropriately to give an even better lower bounds.
\section{\bf Defining Weakly Divisible Rings}\label{Defining}

\bigskip

In this section, we look at Binary Rings associated to {\em weakly divisible polynomials} as defined in \cite{bhargava2022squarefree} and \cite{BSW-monicsquarefree}. We define weakly divisible ring using the binary ring and associate it to said polynomial.

We recall the following theorem from \cite{MM-ONEFINE} defining Binary Rings.
\par 
Let 
\begin{equation*}
  f(X,Y):= a_nX^n+a_{n-1}X^{n-1}Y +\cdots +a_0Y^n
\end{equation*} 
denote a binary form of degree $n$. Let $a_n\neq 0$ and suppose that $\delta$ denotes the image of $X$ in the algebra $\Q[X]/(f(X,1)).$ 
\begin{definition}When $a_n\neq 0,$ we define
\begin{equation}\label{basis-binary}
    R_f:=\Z\ncp{1,\: a_n \delta , \: a_n\delta^2+a_{n-1}\delta,\:\cdots, \:\sum_{i=0}^{k-1}a_{n-i}\delta^{k-i},\:\cdots,\:\sum_{i=0}^{n-2}a_{n-i}\delta^{n-1-i}}.
\end{equation}

We set 
\begin{align}
    \ncp{B_0,B_1,\cdots,B_{n-1}}:=& \ncp{1,\: a_n \delta , \: a_n\delta^2+a_{n-1}\delta,\:\cdots, \:\sum_{i=0}^{k-1}a_{n-i}\delta^{k-i},\:\cdots,\:\sum_{i=0}^{n-2}a_{n-i}\delta^{n-1-i}}\\
    \nonumber \textit{ that is,}\\
    B_0:=&1\\
    B_1:=&a_n\delta+a_{n-1}\\
    \nonumber&\cdots\\
    \nonumber&\cdots\\
    \nonumber&\cdots\\
    B_k:=&a_n\delta^k +a_{n-1}\delta^{k-1}+\cdots+a_{n-k+1}\delta+a_{n-k}\\
    \nonumber&\cdots\\
    B_{n-1}:=&a_n\delta^{n-1} +a_{n-1}\delta^{n-2}+\cdots+a_{2}\delta +a_1
\end{align}
\end{definition}
\begin{remark}
    We refer to this basis as the canonical basis attached to $f.$ 
\end{remark}

\begin{theorem}
    When $f$ is integral(i.e. $f$ is a binary form of degree $n$ with integer coefficients),  $R_f$ is a ring of rank $n$ over $\Z.$
\end{theorem}
\begin{definition}
    We define $I_f$ as the (fractional) ideal class generated by $(1,\delta)$ over $R_f,$ when $f$ is integral.
\end{definition} 
\begin{definition}
    When $f$ is integral, $R_f$ is known as the \textbf{binary ring} associated to the binary form $f.$
\end{definition}
These are a few properties of binary rings.
\begin{proposition}
     Properties of $R_f$ (when $f$ is integral):
    \begin{enumerate}
        \item \begin{equation}
            \Zdisc(R_f)=\disc(f).
        \end{equation} 
        \item If $\delta$ is invertible, and $f$ is primitive, then
        \begin{equation}
            R_f:=\Z[\delta] \cap\Z[\delta^{-1}].
        \end{equation}
        \item If $f$ is primitive, $I_f$ is invertible in $R_f.$
        \item Both $R_f$ and $I_f$ are invariant under the natural $GL_2(\Z)$ action on binary forms of degree $n.$ In particular, for $\delta\in \Qbar\backslash \Q$ this means that if $\lambda=\frac{a\delta+b}{c\delta+d}$ with $a,b,c,d\in \Z$ and $ad-bc=\pm1$ then
        $$\Z[\delta]\cap\Z[\delta^{-1}]=\Z[\lambda]\cap \Z[\lambda^{-1}].$$
    \end{enumerate}
\end{proposition}

You can find the proof of all of these statements in \cite{MM-ONEFINE}. Some of them are previously known, but are only relevant to this paper for the sake of context.
\begin{remark}
    Note that one can write down the multiplication table for the defining basis of $R_f$ in terms of coefficients of $f$ explicitly. Using this table to define binary rings, one can give a definition for $R_f$ without the need for the condition ``$a_n\neq 0$".
\end{remark}
We will only need the following part of the multiplication table of the defining basis of $R_f$ ($\cref{basis-binary}$), which is easily verified.
\begin{lemma}\label{definingtableweakly}
Let $f(X,Y):= a_nX^n+a_{n-1}X^{n-1}Y +\cdots +a_0Y^n$ denote a binary form of degree $n,$ and let $\ncp{B_0,B_1,\cdots,B_{n-1}}$ denote the canonical basis for $R_f$ associated to $f$ as in \cref{basis-binary}.
Then, we have
    \begin{align}\label{multiplcation-Table-needed}
        B_{n-1}\cdot B_{n-i}= -a_0\cdot B_{n-i-1} +a_i\cdot B_{n-1}
    \end{align}
\end{lemma}
\begin{proof}
    We make note of the fact, $B_{n-1}=-a_0\delta^{-1}.$ \\
    Thus,
   \begin{align*}
       &B_{n-1}\cdot B_{n-i}=-a_0\delta^{-1} \cdot (\sum_{j=0}^{n-i} a_{n-j}\delta^{j})\\
       &=-a_0\cdot \sum_{j=0}^{n-i} a_{n-j}\delta^{j-1}\\
       &=-a_0 \cdot (B_{n-i-1} +\frac{a_i}{\delta})\\
       &=-a_0\cdot B_{n-i-1} + a_i\cdot(\frac{-a_0}{\delta})\\
       &=-a_0\cdot B_{n-i-1} +a_i\cdot B_{n-1}.
   \end{align*}
\end{proof}
\begin{definition}
We say a binary form $f(x,y)\in \Z[x,y]$ of degree $n$ is weakly divisible by $m\in \Z_{>0}$ if there exists an $\ell \in \Z$
such that
\begin{align*}
    f(l,1) &\equiv 0\bmod{m^2}\\
    \frac{\partial f}{\partial x} (l,1) &\equiv 0\bmod m
\end{align*}
 When appropriate we say $f$ is weakly divisible by $m$ at $l$.
\end{definition}
\begin{remark}
    We use ``weakly divisible" as these polynomials are defined in \cite{bhargava2022squarefree} and \cite{BSW-monicsquarefree} and the main parts of these papers is about establishing an upper bound for the number of polynomials which are weakly divisible by $m$ and bounded height, for all large $m.$
\end{remark}
\begin{definition}\label{notation $f_l$}
    Given a binary form $f$ we set $f_l=f_l(x,y):=f(x+ly,y).$
\end{definition}
The condition that ``$f$ is weakly divisible by $m$ at $l$" is equivalent to having integers $a_0,a_1,\cdots,a_n$ such that \begin{equation*}
    f_l=a_nx^n+a_{n-1}x^{n-1}y+\cdots+a_{2}x^2y^{n-2} +ma_{1}xy^{n-1}+m^2a_0y^n.
\end{equation*}

For convenience sake, we define $\binom{n}{a}:=0$ if $n\ge 0$ and $a<0.$

\begin{remark}\label{l-matrix of tranformation}
We also mention a useful matrix of transformation for base change from the canonical basis of $R_f$ associated to $f(X,Y)$ to the canonical basis of $R_{f_l}$ associated to $f_l.$

 We let $f(X,Y)=A_0 X^n+A_1X^{n-1}Y+\cdots+A_nY^n$ and  $\ncp{B_0,B_1,\cdots,B_{n-1}}$ be the canonical basis of $R_f$ associated to $f(X,Y)$ and  $\ncp{C_0,C_1,\cdots,C_{n-1}}$ be the canonical basis of $R_{f_l}$ associated to $f_l,$ then we have
\begin{align*}
    \begin{matrix}
        \ncp{C_0,C_1,\cdots,C_{n-1}}= \\ \\ \\ \\ \\ \\ \\ \\ \\  
    \end{matrix}\begin{matrix}
        \ncp{B_0,B_1,\cdots,B_{n-1}}\\ \\ \\ \\ \\ \\ \\ \\ \\ 
    \end{matrix}\begin{bmatrix}
    1 & \binom{n-1}{1} l A_0 & \binom{n-1}{2} l^2 A_0 &\cdots & \binom{n-1}{k-1}l^{k-1}A_0&\cdots &\binom{n-1}{n-1}l^{n-1}A_0\\
    0 & 1 & \binom{n-2}{1}l & \cdots &\binom{n-2}{k-2}l^2&\cdots &\binom{n-2}{n-2}l^{n-2}\\ 
    0&0&1&\cdots &\binom{n-3}{k-3}l^{k-3}&\cdots&\binom{n-3}{n-3}l^{n-3}\\
    \vdots&\vdots &\vdots& \ddots&\vdots&\ddots &\vdots \\
    0 &0&0&\cdots& \binom{n-k'}{k-k'}l^{k-k'}&\cdots & \binom{n-k'}{n-k'}l^{n-k'}\\
    \vdots&\vdots &\vdots& \ddots&\vdots&\ddots &\vdots \\
    0 & 0 &0 &\cdots & \binom{1}{k-n+1}l^{k-n+1}&\cdots & l\\
    0 & 0 &0 &\cdots & \binom{0}{k-n}l^{k-n}&\cdots & 1
\end{bmatrix}
\end{align*}
The proof is given in the appendix.
\medskip\par

\end{remark}
Given $f$ a binary form, we let $\ncp{B_0,B_1,\cdots,B_{n-1}}$ be the canonical basis of $R_{f}$ associated to $f$ as in \cref{basis-binary}.  
\begin{definition}
    We define,
    \begin{equation*}
        R'_{(f,m)}:=\Z\ncp{B_0,B_1,\cdots,B_{n-2},\frac{B_{n-1}}{m}}.
    \end{equation*}
    
\end{definition}
\begin{theorem}\label{Weakly-Divisible-Defining Theorem}
If $f$ is weakly divisible by $m$ at $l,$ then $R'_{(f_l,m)}$ is a ring.
\end{theorem}
\begin{proof}
    Let $\ncp{B_0,B_1,\cdots,B_{n-1}}$ denote the canonical (old) basis of $R_{f_l}$ associated to $f_l.$ Since, $f$ is weakly divisible by $m$ at $l,$ we can write 
    $$
    f_l=a_nx^n+a_{n-1}x^{n-1}y+\cdots+a_{2}x^2y^{n-2} +ma_{1}xy^{n-1}+m^2a_0y^n,
    $$ 
    where $a_i,m\in \Z.$
    
    We will show that $R'_{(f_l,m)}$ is a ring, by showing that product of any two elements in the (new) basis given by $\ncp{B_0,B_1,\cdots,B_{n-2},\frac{B_{n-1}}{m}}$ (basis for $R'_{(f_l,m)})$ is in the $\Z$-span of itself. We will achieve this by comparing the product of every two elements in the old basis with the product of the corresponding two elements in the new basis.)
    
    We note that $B_i \cdot B_j$ for $i,j\neq n-1$ can be written as a $\Z$-linear combination of $\ncp{B_i}$. This follows directly from the fact that the $\Z$-span of $\ncp{B_i}$ forms a ring (the prodigal binary ring) (we can also refer to the multiplication tables given in \cite{Gaurav} or section 2.1 in \cite{MM-ONEFINE}). 
    
    It immediately follows that $B_i \cdot B_j$ for $i,j\neq n-1$ can also be written as a $\Z$-linear combination of our new basis $\ncp{B_0,B_1,\cdots,B_{n-2},\frac{B_{n-1}}{m}}$ (which only differs from the original basis at the index $n-1$) by simply replacing $B_{n-1}$ with $m\cdot \frac{B_{n-1}}{m}$ in the multiplication table of the old basis.
    
    On the other hand, \cref{definingtableweakly} immediately tells us that, when $i\neq 1$
    \begin{align*}
        \frac{B_{n-1}}{m}\cdot B_{n-i} = -ma_0\cdot B_{n-i-1}  +a_i\cdot \frac{B_{n-1}}{m} \\
    \end{align*} 
    and
    \begin{align*}
        \frac{B_{n-1}}{m}\cdot \frac{B_{n-1}}{m} = -a_0\cdot B_{n-2} +a_1\cdot \frac{B_{n-1}}{m}.
    \end{align*}
    Thus, the products of elements in this new basis are $\Z$-linear combinations of the same new basis. It follows that $R'_{(f_l,m)}$ is in-fact a ring.
\end{proof}
\begin{definition}
    We say $R'_{(f_l,m)}$ is the weakly divisible ring (at $l$ with respect to $m$) associated to $f$, when $f$ is weakly divisible by $m$ at $l$. When appropriate we will also represent this ring as $R'_{(f,m,l)}.$
\end{definition}
\begin{remark}
    Every binary ring is a weakly divisible ring at every value with respect to 1.
\end{remark}
\begin{remark}
   We note that weakly divisible rings may also defined by multiplication tables to avoid dependence on a condition like ``$a_n\neq 0$ or $a_0\neq 0$".\end{remark}





\section{\bf Effective injectivity of the map: $$(f,m,l)\longrightarrow R'_{(f_l,m)}$$}

\bigskip
One of the issues we must resolve to count rings and number fields via the above correspondence is the issue of distinct polynomials offering up the same rings or number fields.
We resolve this using an approach similar to the one in \cite{BSW-monicsquarefree}.
\begin{remark}
    The result we get here is less than ideal and improvement on characterization of subsets on which this map will be injective or even some characterizations of subsets on which the map $f\longrightarrow R_f$ is injective, will further improve the lower bound. In a future paper, we will write down a conjecture qualifying what we expect to happen in ideal circumstances in such a situation.
\end{remark}
Before we dive into the details, we observe the following lemma which talks about injectivity for a fixed $m$ and $f.$
 \begin{lemma}\label{same-m}
     If $f$ is weakly divisible by $m$ at $a$ and $b$ then 
     $$
     R'_{(f_a,m)}= R'_{(f_b,m)}\iff a\equiv b\bmod m.
     $$
 \end{lemma}
 \begin{proof}
 Suppose $R'_{(f_a,m)}= R'_{(f_b,m)}$.
 Let the canonical basis for $R_f,$ the binary ring associated to $f$, be denoted by $$
 \ncp{1,B_1,B_2,\cdots,B_{n-2},B_{n-1}}.
 $$
 Then, we observe that
 $$
 A:=\ncp{1,B_1,B_2,\cdots, B_{n-2}, \frac{B_{n-1}+aB_{n-2} + r_1 B_{n-3}+\cdots+r_{n-2}}{m}}
 $$
 is a basis for $R'_{(f_a,m)}$ and 
 $$
 B:=\ncp{1,B_1,B_2,\cdots, B_{n-2}, \frac{B_{n-1}+bB_{n-2} + s_1 B_{n-3}+\cdots+s_{n-2}}{m}}
 $$
 is a basis for $R'_{(f_b,m)}$ where $r_i,s_i\in \Z.$ 
 These claims follow directly from Remark \cref{l-matrix of tranformation} as the canonical basis for $R_f$ (associated to $f$) and that for $R_{f_l}$ (associated to $f_l)$ is an upper triangular (integer and unipotent) matrix away from one another with the $1$s along the diagonal and the entry at $(n-1,n)$ is $l$. 
In other words, the entry in the matrix of basis change(from $A$ to $B$) has the entry $\frac{a-b}{m}$ at the $(n-1,n)$ place.

Thus, if $R'_{(f_a,m)}=R'_{(f_b,m)}$, then the matrix of transformation for associated bases has to be an invertible integer matrix, that is, $\frac{a-b}{m}\in \Z$ or $a\equiv b \bmod m.$

To see the ``if" part we again note that the exact matrix for the canonical basis transformation 
$$
R_f\longrightarrow R_{f(x+l)} 
$$ 
is given by an upper triangular matrix whose entries beyond the first row are all independent of the coefficients of $f$, see \cref{l-matrix of tranformation}. Clearly, the matrix transformation  
$$
R_f\longrightarrow R_{f(x+m)}
$$ 
is congruent to the identity 
modulo $m$. Now, if $a \equiv b \bmod m$ then $(f_a)_b=f_a(x+mk)$. We thus see that the change of base matrix from $R_{f_a}$ to $R_{f_b}$ is identity modulo $m.$ Thus, dividing the last column of this matrix by $m$ and multiplying the last row of this matrix by $m$ still gives us an invertible integer matrix. This new matrix will be the matrix of change of basis for 
$$
R'_{(f_a,m)}\longrightarrow R'_{(f_b,m)}.
$$  
The result follows. 
 \end{proof}
\subsection{Fundamentals of Minkowski Reduction}
Let $\ncp{v_1,v_2,\cdots,v_n}$ denote a real basis for $\R^n.$
We perform a Gram-Schmidt reduction on the basis using some canonical distance form. We write 
$$\ncp{v_1,v_2,\cdots,v_n}= \ncp{B}M [(t_1,t_2,...,t_n)]$$ where $B$ is an orthonormal ordered basis of vectors, all of which are of the same size. $M$ is an upper triangular uni-potent matrix. And $[(t_1,t_2,..,t_n)]$ denotes a diagonal matrix with $t_i$ placed in the $(i,i)^{th}$ place.

Let $v_i'$ denote the projection of $v_i$ to the space orthogonal to the space spanned by $\{v_1,v_2,\cdots, v_{i-1}\}$. Then, Gram-Schmidt process forces  
\begin{equation}\label{t_i}
    |t_i| =||v_i'||.
\end{equation}
We note that for each $1 \le i \le n$, and any choice of coefficients $a_1, \cdots, a_i$ we have
$$
||a_iv_i + \sum_{j=1}^{i-1}a_jv_j||\ \ge\ |a_i| |t_i|.
$$
This follows by looking at the component of the given vector (on RHS) along $v'_i.$ 
Moreover, for each $1 \le i \le n$, there exist $b_1, \cdots, b_{i-1}\in \Z$ so that
$$
||v_i + \sum_{j=1}^{i-1}b_jv_j||\le |t_i| +|t_{i-1}|+\cdots+|t_1|.
$$

One can simply choose $b_j$ such that the $\sum_{j=1}^{i-1}b_jv_j$ approximates the vector $v_i-v'_i$ which clearly lies in the space spanned by $\ncp{v_1,v_2,\cdots,v_{i-1}}.$

Now if $t_{i+1}/{t_i}>2$ for all $i$, one can see that inductively the $i^{th}$ vector in the Minkowski reduced basis of the lattice spanned by $\{v_1,v_2\cdots, v_n\}$ will have the form $\pm v_i +\sum_{j=1}^{i-1}b_jv_j.$

In fact, having
$$
t_{i}/{t_{i-1}}>\sqrt{1+\frac{1}{i}} \textit{ for all } 1\le i\le n
$$ 
is sufficient to conclude this.

Furthermore, if we just know that,
\begin{equation}
    \frac{t_i}{t_2}\ge 2 \textit{ for all } 3\le i\le n \textit{ and  } \frac{t_2}{t_1}\ge 2
\end{equation}
then, the first two vectors (which will be unique up to sign) in the Minkowski reduced basis for the lattice $\ncp{v_1,v_2,\cdots,v_n}$ will have the above form. 
That is 
$$
\{\pm v_1 , \pm (v_2 +a\cdot v_1)\}
$$
will be the first two elements in the Minkowski reduced basis for this lattice spanned by $\ncp{v_1,v_2,\cdots,v_n}.$

\begin{definition}
    If $\ncp{v_1,v_2,\cdots,v_n}$ is a basis for $\R^n$ and $v_i'$ denotes the projection of $v_i$ to the space orthogonal to the space spanned by $\{v_1,v_2,\cdots, v_{i-1}\}$ and $t_i=||v_i'||,$ then we say   $\ncp{v_1,v_2,\cdots,v_n}$ is {\em Normally Minkowski Reduced} if $t_i$ satisfy 
    \begin{equation}\label{first 2 minkow}
    \frac{t_i}{t_2}\ge 2 \textit{ for all } 3\le i\le n \textit{ and  } \frac{t_2}{t_1}\ge 2
\end{equation}
\end{definition}
Our discussion above gives us the following lemma:
\begin{lemma}
    If $\ncp{B_1,B_2,\cdots,B_n}$ is {\em Normally Minkowski Reduced}, then
    \begin{enumerate}
        \item $B_1$ is the unique smallest vector in $L$ up to sign.
        \item $\exists a\in \Z$ such that, $B_2+a\cdot B_1$ is the smallest vector in $L$ which is not in $\Z\cdot B_1$ up to sign.
    \end{enumerate}
\end{lemma}
\subsection{Reduced-m-Polynomials}
\begin{definition}
    We say a tuple $(f,m,l)$ is a Reduced-$m$-Polynomial (at $l$) if
    \begin{itemize}     
        \item The canonical basis for $R'_{(f_l,m)}$ is Normally Minkowski Reduced when seen in $R'_{(f_l,m)}\otimes \R\simeq \R^r\oplus\C^s$ under the canonical norm.
    \end{itemize}
\end{definition}
\begin{remark}
     We will ignore the $l$ part as $(f,m,l)$ is a Reduced-$m$-polynomial (at $l$) $\iff$ $(f,m,r)$ is a Reduced-$m$-Polynomial (at $r$) for any real value $r.$ This is easy to see from the fact that the matrix in \cref{l-matrix of tranformation} is upper triangular and hence will leave relative sizes of normal components unchanged. 
\end{remark}

\begin{theorem}\label{distinct-rings}
    Given $f,g$ polynomials of degree $n\ge 4$ such that $f$ is a Reduced-$m$-polynomial and $f$ is weakly divisible by $m$ at $e$ and $g$ is a Reduced-$m'$-polynomial and $g$ is weakly divisible by $m'$ at $d$ ($m,m'\ge 1$) with $R'_{(f_e,m)}= R'_{(g_d,m)}$ then 
    $$
    m=m' \textit{ and }\exists r\in \Z: g(x)=f(x+mr-d+e).
    $$ 
\end{theorem}
\begin{proof}
    Let $\beta$ and $\alpha$ denote roots of $f$ and $g$ respectively ($\in\R^r\oplus\C^s$). Let $a_0$ and $b_0$ denote the positive leading coefficients of $f$ and $g$ and $a_1$ and $b_1$ denote the second leading coefficient of $f$ and $g$ respectively. 
    
    We know that the smallest two elements of any Normally Minkowski Reduced basis are unique. Since $R'_{(f_e,m)}=R'_{(g_d,m')},$ comparing the smallest two elements elements in these, we get
    $$
    \ncp{1,a_0\cdot\beta}\begin{bmatrix} a & b\\0 & c \end{bmatrix} =\ncp{1,b_0\cdot \alpha}
    $$
    where $a=\pm 1$ and $c=\pm1.$ Clearly, $a=1$ and thus it follows that
    $$
    \alpha=\frac{\pm a_0 }{b_0}\cdot \beta +\frac{b}{b_0}.
    $$
    $$
    \frac{f(\frac{\pm a_0 }{b_0}\cdot x +\frac{b}{b_0})}{a_0}=(\pm \frac{a_0}{b_0})^n \frac{g(x)}{b_0}.
    $$

    This immediately tells us that the matrix of transfer from $\ncp{1,a_0\beta,a_0\beta^2+a_1\beta}$ to $\ncp{1,b_0\alpha,b_0\alpha^2+b_1\alpha}$ is 
    $$
    \begin{bmatrix} 
    1 & b & * \\
    0& c& *\\
    0& 0& \frac{a_0}{b_0}
    \end{bmatrix}.
    $$
    Since $n\ge 4$ the above matrix must also be integral and invertible.
    It follows that $a_0=b_0$ and thus $\disc(f)=\disc(g)$.
    
    Since $R'_{(f_e,m)}=R'_{(g_d,m')}$ we also have 
    $$
    \frac{\disc(f)}{m^2}=\frac{\disc(g)}{m'^2},
    $$ 
    and we immediately get $m=m'.$
    
    Now, without loss of generality, we may assume $c=1$, for if $c=-1$ we may change $g(x)$ to $(-1)^ng(-x).$ We let $l=\frac{b}{b_0}\in \Q.$ Thus, $f(x+l)=g(x).$
    
    We make note of the fact that $f$ and $g$ are weakly divisible by the same value $m$.
    
    Observing the matrix of transformation from $R_{f_e}$ to $R_{g_d}$ given using \cref{l-matrix of tranformation}, we notice that the entry in the $(n-1)^{th}$ row and $n^{th}$ column of the matrix of transfer for the canonical bases of $R_{f_e}$ to $R_{g_d}$ is $d-e+l$.
    
    Thus, the entry in the $(n-1)^{th}$ row and $n^{th}$ column of the matrix of transfer for the canonical bases of $R'_{(f_e,m)}$ to $R'_{(g_d,m)}$ is $\frac{d-e+l}{m}$. Thus $r=\frac{d-e+l}{m}$ must be an integer. It follows that $\exists r\in \Z : g(x)=f(x+mr-d+e).$
\end{proof}
\section{\bf Rings and Orders when ordered by discriminant}

\bigskip
\begin{definition}
    \begin{equation*}
        W(s:t):=\Big\{\sum_{i=0}^n a_ix^{n-i}\in \Z[x]:0\le a_1\le \frac{s}{2}\le a_0\le s,  |a_i|\le st^i  \text{ for all } 2\le i\le n \Big\}
    \end{equation*}
\end{definition}
\begin{remark}
    The above set is a subset of a fundamental domain for action of integers by translation on polynomials. That is, for any polynomial $f$ there exists at most one integer $l$ such that $f(x+l)\in W(s:t).$
\end{remark}
\begin{lemma}\label{minkow}If $f$ is a real monic polynomial, then there exists a $\rho_f\in \R$ (continuously varying with $f$) such that $\lambda\rho^{n} f(\frac{x}{\rho})$ is a Reduced-$1$-polynomial for all $\rho\ge \rho_f$ and $\lambda\ge 1$.

If $\rho \ge m^{1/(n-2)}\rho_f$ and $\lambda\ge 1$ the polynomial $\lambda\rho^{n} f(\frac{x}{\rho})$ is a Reduced-$m$-polynomial.
\end{lemma}
\begin{proof}
We follow the argument of Bhargava-Shankar-Wang proving in Lemma 5.2 in\cite{BSW-monicsquarefree}.

Set 
$$
f(x)=x^n+a_{n-1}x^{n-1} +\cdots+ a_{k} x^{n-k}+\cdots+ a_0
$$
and set $a_n=1.$

We perform a Gram-Schmidt reduction of the basis for $R_f$ using the canonical distance form on $R_f\bigotimes \R\simeq \R^r\bigoplus \C^s$($r+2s=n$). We write 
$$
\ncp{1,\:  \delta , \: \delta^2+a_{n-1}\delta,\:\cdots, \:\sum_{i=0}^{k-1}a_{n-i}\delta^{k-i},\:\cdots,\:\sum_{i=0}^{n-2}a_{n-i}\delta^{k-i}}= \ncp{B}M [(t_1,t_2,...,t_n)]
$$ 
where $B$ is an orthonormal ordered basis of vectors, $M$ is an upper triangular unipotent 
matrix, and $[(t_1,t_2,..,t_n)]$ denotes a diagonal matrix with those entries along the diagonal.
Now we note that if \begin{align*}
    \ncp{1,\: \delta , \:& \delta^2+a_{n-1}\delta,\:\cdots, \:\sum_{i=0}^{k-1}a_{n-i}\delta^{k-i},\:\cdots,\:\sum_{i=0}^{n-2}a_{n-i}\delta^{n-1-i}}\\
    &= \ncp{B}M [(t_1,t_2,\cdots,t_k,\cdots,t_n)] 
    \end{align*}    
then, using \cref{t_i}, we see that 
\begin{align*}
    \ncp{1,\: \lambda(\rho \delta) , \: &\lambda((\rho\delta)^2+a_{n-1}\rho\delta),\:\cdots, \:\lambda(\sum_{i=0}^{k-1}a_{n-i}(\rho\delta)^{k-i}),\: \cdots,\:\lambda(\sum_{i=0}^{n-2}a_{n-i}(\rho\delta)^{n-1-i}})\\
    &= \ncp{B}M_\rho [(t_1,\lambda \rho t_2,\cdots, \lambda\rho^{k-1}t_k,\dots ,\lambda\rho^{n-1} t_n)].
\end{align*}
Thus, for each polynomial $f$ one may find $\rho_f$ which is a continuous function of $f$ such that $\lambda \rho^n f(\frac{x}{\rho})$ is Minkowski reduced for all $\lambda\ge 1$ and $\rho\ge\rho_f$. One can simply take 
$$
\rho_f:=\max\{\max_{i\ge 3}\{(\frac{2t_{2}}{t_{i}})^{1/(i-2)}\}, \frac{2t_1}{t_2}\}.
$$
Furthermore, we note that translating the polynomial changes the canonical basis by an upper triangular matrix. 
\medskip\par
Thus, if $g(x)=\lambda \rho^{n}f(\frac{X}{\rho}),$ then the canonical basis of $R'_{(g_l,m)}$ after Gram-Schmidt process will look like 
$$
\ncp{B}M_{\rho,m,l} [(t_1,\lambda \rho t_2,\cdots, \lambda\rho^{k-1}t_k,\cdots,\lambda\rho^{n-2}t_{n-1},\frac{\lambda\rho^{n-1}}{m} t_n)]
$$
where $M_{\rho,m,l}$ is a uni-potent upper triangular matrix. 

It follows that for $\rho\ge m^{1/(n-2)}\rho_f$  and $\lambda \ge 1,$ $\lambda \rho^{n}f(\frac{X}{\rho})$ is Reduced-$m$-polynomial.
\end{proof}

\begin{definition}
\begin{equation*}
    W(s:t:m):=\{(f,l):f\in W(s:t), 0\le l<m, f \textit{ is weakly divisible by } m  \textit{ at } l\}
\end{equation*}
\end{definition} with the understanding that each element $(f,l) \in W(s:t:m)$ where $f$ is Reduced-$m$-Polynomial will give distinct rings, see Theorem \cref{distinct-rings}.

\begin{lemma}
    Given $a_0,a_1,\cdots,a_{n-2},l,m$ there is a unique choice of $a_{n-1}$ and $a_n$ satisfying \begin{itemize}
        \item $$B+1 \le a_{n-1}\le B+m$$
        \item $$C+1 \le a_{n-1}\le C+m^2$$
        \item $$a_0x^n+a_1x^{n-1}+\cdots+a_n \textit{ is weakly divisible by $m$ at $l.$} $$
    \end{itemize}
\end{lemma}
\begin{proof}
    Follows from the definition directly.
\end{proof}
Thus, we count elements in $W(s:t:m)$ by choosing $l, a_0,\cdots a_{n-2}$ in generality and then use the above lemma to choose the final two coefficients of the polynomial. 
\begin{align*}
    &|W(s:t:m)|\\
    =&m\cdot (\frac{s}{2}+O(1))^2\cdot (2st^2+O(1))\cdots(2st^k+O(1)\cdots
    (2st^{n-2}+O(1))\cdot(\frac{2st^{n-1}}{m}+O(1))\cdot (\frac{2st^n}{m^2}+O(1))\\
    =&\frac{2^{n-3}s^{n+1}t^{\frac{n(n+1)}{2}-1}}{m^2}(1+O(\frac{m^2}{st^n} +\frac{1}{s})).
\end{align*}

Let $B_\epsilon$ denote a compact set in 
$$
\{f\in\R[x]: H(f) \le 1, \Zdisc(f) \neq 0 \}
$$
such that 
$$
Vol(B_\epsilon)\ge (1-\epsilon)Vol(\{f:H(f)\le 1\}).
$$
Then for this region we can construct a $\rho_B$ such that all polynomials 
$f$ with height $\ge \rho_B$ 
satisfying 
$$
H(f)^n f(\frac{X}{H(f)})\in B_\epsilon
$$
are Reduced-$1$-Polynomials. 
From the previous lemma, we can say that polynomials of this type with height 
$$
\ge m^{1/(n-2)}\rho_B
$$
are Reduced-$m$-Polynomials. Furthermore, we may consider $B_\epsilon$ to be a finite union of disjoint boxes $B_i$ (depending on epsilon).

Thus, correspondingly, if $d_{i,r}$ are the dimensions of the box $B_i,$ we define $W(s:t:m)_{B_i}$ to be those polynomials that satisfy 
$$
\frac{H(f)^n}{s}f(\frac{X}{H(f)})\in B_i
$$ 
and 
$$
\frac{f^{(2)}(0)}{2!} \ge s(\rho_{B})^{n-2}m.
$$
The latter condition, $\frac{f^{(2)}(0)}{2!} \ge s(\rho_{B})^{n-2}m,$ ensures that polynomials we count are of `large enough'(see Lemma \cref{minkow}) height to be a Reduced-$m$-Polynomial. 
\medskip\par
Thus,
\begin{align*}
    |W(s:t:m)_{B_i}|= & m\cdot (\frac{s}{2} +O(1))(\frac{s}{2}d_{i,1}+O(1))\cdot (2st^2d_{i,2}+O(1))\cdots (2st^k d_{i,k}+O(1))\cdots(2st^{n-3} d_{i,n-3}+O(1)) \\ 
    & \quad\quad \cdot(2st^{n-2}d_{i,n-1}-2s(\rho_B)^{n-2}m+O(1)) \:\:\cdot(\frac{2st^{n-1}d_{i,n-1}}{m}+O(1))\cdot (\frac{2st^nd_{i,m}}{m^2}+O(1))\\
    =&\frac{2^{n-3}Vol(B_i)s^{n+1}t^{\frac{n(n+1)}{2}-1}}{m^2}(1+O_\epsilon(\frac{m^2}{st^n} +\frac{1}{s} + \frac{m}{t^{n-2}})).
\end{align*}
\begin{definition}
\begin{equation*}
    W(s:t:m)^{Red}:=\{(f,l)\in W(s,t,m): f \textit{ is a Reduced-}m\textit{-Polynomial.}  \}
\end{equation*}
\end{definition}

\begin{remark}\label{RedPolyDesity}
Thus,
  \begin{align*}
    |W(s:t:m)^{Red}| &\ge \sum_i |W(s:t:m)_{B_i}|\\ &\ge (1-\epsilon)\cdot\frac{2^{n-3}s^{n+1}t^{\frac{n(n+1)}{2}-1}}{m^2}(1+O_\epsilon(\frac{m^2}{st^n} +\frac{1}{s}+\frac{m}{t^{n-2}})).
\end{align*}  
\end{remark}

Since, the discriminant of the polynomials in $W(s:t)$ is $\ll_n s^{2n-2} t^{n(n-1)},$ the discriminant of the weakly divisible ring in $W(s:t:m)$ it corresponds to will be 
$$
\ll_n \frac{s^{2n-2}t^{n(n-1)}}{m^2}.
$$

We set $X=\frac{s^{2n-2}t^{n(n-1)}}{m^2}$ i.e 
$$
t=(\frac{m^2X}{s^{2n-2}})^{\frac{1}{n(n-1)}}.
$$

Let $S(X,s,m)$ denote the number of rings with discriminant $\ll_n X$ counted by $W(s:t:m)^{Red}$ for the above value of $t.$

It follows that
\begin{theorem}For $n\ge 4,$
    \begin{equation}
        S(X,s,m)\gg_n (1-\epsilon) \frac{X^{\frac{1}{2}+\frac{1}{n}}\cdot s^{\frac{2}{n}}}{m^{1-\frac{2}{n}}}
    \end{equation}
provided that\begin{enumerate}
    \item $s\gg_\epsilon 1$
    \item\label{ti} $X\gg_{\epsilon} s^{2n-2}\cdot m^{n-1+\frac{2}{n-2}}$ (Condition corresponding to $t^{n-2}\gg_\epsilon m$ - required for injectivity).
    \item $X\gg_{\epsilon} s^{n-1}\cdot m^{2n-4}$ (Condition corresponding to $st^n\gg_\epsilon m^2$-required for there being enough polynomials which are weakly divisible by $m$)
\end{enumerate}   
\end{theorem}
To count ultra-weakly divisible rings which are weakly divisible by rings we wish to choose $s$ appropriately. Note that within the constraints for a fixed $X$, the rings counted will all be distinct for any choice of $s$ and then followed by any choice of $m.$ This observation requires us to use \cref{distinct-rings} which in turn requires $n\ge 4.$

We will restrict ourselves to a region where \cref{ti} provides a stronger bound for $s$.
This amounts to the region
$$
(\frac{X}{m^{2n-4}})^{\frac{1}{n-1}}\ge (\frac{X}{m^{n-1+\frac{2}{n-2}}})^{\frac{1}{2n-2}} \iff X\ge m^{3n-7-\frac{2}{n-2}}.
$$ 

In this region, the largest value we may pick for $s$ is $(\frac{X}{m^{n-1+\frac{2}{n-2}}})^{\frac{1}{2n-2}}.$ We pick this value to get that the number of distinct rings weakly divisible by $m$ and with discriminant less than $X$ is 
$$
\gg \frac{X^{\frac{1}{2}+\frac{1}{n-1}}}{m^{1-\frac{(n-3)}{(n-1)(n-2)}}}.
$$

We can sum over all possible $m$ satisfying the condition $m\ll_{\epsilon} X^{1/(3n-7-\frac{2}{n-2})}$ to get that the number of weakly divisible rings with discriminant less than $X$ to be
$$
\gg_{n} X^{\frac{1}{2}+\frac{1}{n-\frac{4}{3}}}.
$$

This proves \cref{1}.
\begin{remark}
    Looking at the region where the other condition is dominant gives a lower bound of the same order.
\end{remark}
\begin{remark}
    If we had injectivity intrinsically, say if injectivity conjecture from \cite{Gaurav}(and a future paper) were true, then we would get
    $$
    \bigg\{\substack{\text{Rings or Orders} \\ \text{with discriminant}\le X}\bigg\}\gg X^{\frac{1}{2}+\frac{1}{n-\frac{3}{2}}}.
    $$
    \end{remark}
\begin{remark}\label{Matt-Baker-Argument}
    To restrict oneself to orders in number fields with Galois group $S_n$ we can look at only those points which are irreducible modulo $2$ (thus, forcing the polynomial to be irreducible itself and forcing and $n$-cycle in the Galois Group), which have a factor of degree $p$ for some prime $n/2< p \le n$ modulo 3 (thus,forcing a $p$-cycle in the Galois Group). This completes the argument for $n\neq 6,7,$ as transitive subgroup of $S_n$ containing a $p$-cycle for some $n/2<p\le n$ must be $S_n$. For $n=6,7$ this group is forced to be $A_n.$ We can make sure that the binary forms have all linear factors except for exactly one degree $2$ factor modulo 5. This will imply the existence of a transposition in the Galois Group, which in turn will force the Group to be $S_n.$ We refer to 
    
    \href{https://mattbaker.blog/2014/05/02/newton-polygons-and-galois-groups/}{https://mattbaker.blog/2014/05/02/newton-polygons-and-galois-groups/}

    for an excellent exposition on the matter.
\end{remark}
\section{\bf A lower bound for the number of number fields with bounded discriminant}

\bigskip

We set $n\ge 6.$
\begin{definition}
\begin{equation*}
    W(s:t:m):=\{(f,l):f\in W(s:t),0\le l<m, f \textit{ is weakly divisible by } m \textit{ at } l\}
\end{equation*}
\end{definition}
For small $t$ and $m$ satisfying $m\le t^{n-\epsilon}\le s^{\frac{1}{2}}t^{\frac{n}{2}}$, we will count number of elements $(f,l)$ in $W(s,t,m)$ such that for all primes (in $\Z$) 
\begin{equation}\label{local-cond}
    p\nmid m\Rightarrow p\nmid [\O_{\K_f} : R_f] \textit{ and } p|m \Rightarrow p\nmid \frac{\disc(f)}{m^2}.
\end{equation}
Note that, when $f$ is weakly divisible by $m$ at $l,$

\begin{center}
    $[\O_{\K_f} : R'_{(f_l,m)}]$ divides $[\O_{\K_f} : R_f]$ 
    
    and 

    $[\O_{\K_f} : R'_{(f_l,m)}]^2$ divides $\frac{\disc{f}}{m^2}.$
\end{center} 

Together these imply that the ring counted by $(f,l)$ when $f$ is weakly divisible by $m$ at $l$ and satisfying \cref{local-cond} is the ring of integers of $\K_f$ as they imply that the index of $R_{(f_l,m)}$ in  $\O_{\K_f}$ has no prime factor. Let us refer to these as $NF$ polynomials (as they correspond to distinct number fields) and define
\begin{definition}
    $$
    W(s:t:m)^{NF}:=\{(f,l)\in W(s:t:m): f \textit{ is a $NF$ polynomial}\}
    $$
\end{definition}

In \cite{bhargava2022squarefree} Proposition A.2, it is shown that the ratio of the number of the polynomials over $\sfrac{\Z}{p^2}$ for which $p|[\O_{\K_f} : R_f]$ to all polynomials over $\sfrac{\Z}{p^2}$ is 
$$
d_p:=\frac{1}{p^2}+\frac{1}{p^3}-\frac{1}{p^5}.
$$ 
Thus, when $p\nmid m,$ we have $v_p([\O_{\K_f} : R'_{(f_l,m)}])=v_p([\O_{\K_f} : R_f])$ and thus, the ratio of polynomials over $\sfrac{\Z}{p^2}$ for which $p|[\O_{\K_f} : R'_{(f_l,m)}]$ to all polynomials over $\sfrac{\Z}{p^2}$ is also $d_p.$

We note that, for $p|m$ the condition 
$$
p\nmid \frac{\disc(f)}{m^2}
$$
corresponds to 
$$
p\nmid \frac{f^{(2)}(l)}{2}\textit{ and } p\nmid \disc(f^*)
$$ 
where $f$ is weakly divisible by $p$ at $l$ and $f^{2}(l)$ denotes the second derivative of $f$ at $l$. Moreover, 
$$
f^*=\frac{f(x)-f(l) -f'(l)(x-l)}{(x-l)^2}.
$$
which we see is a polynomial. We see the discriminant as a determinant of Sylvester matrix associated to $f_l$ and note that $p^2|f(l)$ and $p|f'(l)$.   We get 
$$
\frac{\disc(f)}{m^2}\equiv (\frac{f^{(2)}(l)}{2})^2\cdot\disc(f^*)\bmod m
$$ 
from which the result follows. 

\begin{definition}\label{otherlocal}
    \begin{equation}
        c_p:=1-(1-\frac{1}{p})(1-\frac{1}{p}).
    \end{equation}
\end{definition}
Thus, the reciprocal polynomial of $f^*$ will be a degree $n-2$ polynomial with non zero leading coefficient when seen as an element of $\F_p[x].$ The number of polynomials with non zero leading coefficient  is $(p-1)$ (number of choices for leading coefficient) times the number of such monic polynomials, which is $p^n\cdot(1-\frac{1}{p}),$ from Proposition 6.4 in \cite{SQFREE-Prob}.
Which is 
$$
(p-1)\cdot(1-\frac{1}{p})p^n.
$$ 
Dividing by $p^{n+1}$ and subtracting this quantity from 1 we get $c_p.$ By our qualification when $n\ge 6,$ $c_p$ denotes the proportion of polynomials which are weakly divisible by $p$ for which $p^3|\disc(f)$ in the set of all polynomials which are weakly divisible by $p.$

Based on our argument above, for $p|m$ and for a fixed $l$, $c_pp^{n+1}$ is also the number of $(f,l)$ seen modulo $p$ for which $p\nmid \frac{\disc(f)}{m^2}$.

Let 
$$
W(s:t:m:m'):=\{(f,l)\in W(s:t:m): \frac{m'}{(m',m)}|[\O_{\K_f} : R_f], (m',m)|\frac{\disc(f)}{m^2}\}.
$$

Let 
$$
\tau_{m'}:=(\prod_{p|m'}d_p \prod_{p|(m',m)}\frac{c_p}{d_p} )
$$
\begin{lemma}If $m'\ll \min\{s,\frac{st^n}{m^2}\},$ then
    $$
    |W(s:t:m:m')|=\tau_{m'}2^{n-3}\frac{s^{n+1}t^{\frac{n(n+1)}{2}-1}}{m^2}(1+O(\frac{m'}{s} + O(\frac{m^2 m'}{st^n}))).
    $$
\end{lemma}
\begin{proof}
    We see this follows naturally by making cube boxes of dimension $n+1$ with side-length $m'$ for each of the coefficients of $f$ in the region and noting that there are exactly  $\tau_{m'}(m')^{n+1}$ tuples $(f,l)$ of interest in this region.
\end{proof}
Let 
$$
\Theta_M:=\prod_{p< M} \big(1-d_p\big) \prod_{p|m,p<M}\frac{1-c_p}{1-d_p}
$$
We thus have 
\begin{align*}
    |W(s:t:m)^{NF}|=&\sum_{m'\ge 1} \mu(m') |W(s:t:m:m')|\\
    =&\sum_{m'\le M}\mu(m')|W(s:t:m:m')| +O\Bigg(\mvcenter{\sum_{\substack{m'> M\\m'\:\square\textit{-free}}}|W(s:t:m:m')|}\Bigg)\\
\end{align*}
    We will refer to 
    $$
    O\Bigg(\mvcenter{\sum_{\substack{m'> M\\m'\:\square\textit{-free}}}|W(s:t:m:m')|}\Bigg)
    $$
    as the \textit{Error Term} which we will handle later. So, we have
\begin{align*} 
    |W(s:t:m)^{NF}|=&\sum_{m'\le M}\mu(m')\tau_{m'}2^{n-3}\frac{s^{n+1}t^{\frac{n(n+1)}{2}-1}}{m^2}(1+O(\frac{m'}{s})+O(\frac{m^2 m'}{st^n})) + \textit{Error Term}\\
    =&\Theta_M2^{n-3}\frac{s^{n+1}t^{\frac{n(n+1)}{2}-1}}{m^2}(1+M^2O(\frac{1}{s}+\frac{m^2}{st^n})) + \textit{Error Term}
\end{align*}
We note that the \textit{Error Term} counts a subset of tuples of the type $(f,l,m')$ such that $m'^2 |\disc(f)$ and $m'> M$. 

Note that for a given $f\in W(s:t)$ and a fixed $p^k||m$, $f$ being weakly divisible at $l$ modulo $p^k$ implies that $f(x)$ has a double root modulo $p^k$ at $l.$ Since, $f(x)$ can have at most $n$ roots counted with multiplicity modulo $p^k$, it can have at most $n/2$ distinct linear roots modulo $p^k$ with multiplicity strictly greater than $1$. This means that there is at most $(\frac{n}{2})^{\omega(m)}\ll m^\epsilon$ values of $l$ for which $f$ is weakly divisible by $m$ at $l$.
Thus, this sum can be bounded above by the number of tuples $(f,m')$ such that $m'>M$ and $m'^2|\disc(f)$ taken $m^{\epsilon}$ times. Since the coefficients of $f$ are all less than or equal to $st^n$ we may borrow the tail end estimate from \cite{bhargava2022squarefree}(Corollary 6.27) to bound the \textit{Error Term} by
$$
(\frac{(st^n)^{n+1}}{M^{\frac{1}{3}}(st^n)^{-c_n-\epsilon}} + \frac{(st^n)^{n+1}}{(st^n)^{\eta_n-\epsilon}})\cdot m^{\epsilon}
$$
where $\eta_n=\frac{1}{5n}$ and $c_n=0$ if $n$ is odd and is $\eta_n=\frac{1}{88n^6}$ and $c_n=\frac{1}{88n^5}$ as in Corollary 6.27 in \cite{bhargava2022squarefree}.

We pick $M=(st^n)^{3c_n+3\eta_n}.$

We also make note of the fact that $\Theta_M=O(\frac{1}{\log(m)^2}).$

It follows that, provided we have 
$$
m^{2+\epsilon}\le \min\bigg\{ (st^n)^{1-2\eta_n},\frac{(st^n)^{\eta_n}}{t^{\frac{n(n+1)}{2}+1}},(st^{n})^{1-\frac{3}{2}(c_n+\eta_n)}\bigg\}=\frac{(st^n)^{\eta_n}}{t^{\frac{n(n+1)}{2}+1}},
$$
we get
$$
|W(s:t:m)^{tm}|\gg  \frac{s^{n+1}t^{\frac{n(n+1)}{2}-1}}{m^{2+\epsilon}}.
$$

These $(f,l)$ count distinct rings with discriminant $\le \frac{s^{2n-2}t^{n(n-1)}}{m^2}$ when $m^{1+\epsilon} \le t^{n-2}$ since all but an epsilon proportion of these polynomials with height $\gg_\epsilon m^{\frac{1}{n-2}}$ are Reduced-$m$-Polynomials, see \cref{RedPolyDesity}. 

So setting 
$$
X:=\frac{s^{2n-2}t^{n(n-1)}}{m^2} \textit{ or } s=(\frac{m^2X}{t^{n(n-1)}})^{\frac{1}{2n-2}}
$$
we get number fields with discriminant $\le X$ and having a ring of integers weakly divisible by $m$ is 
$$
\gg X^{\frac{1}{2}+\frac{1}{n-1}}\frac{1}{tm^{1-\frac{2}{n-1}-\epsilon}}.
$$
We set $t\simeq m^{\frac{1}{n-2}+\epsilon}$ to make sure that all polynomials except for a proportion of $\epsilon$ (We pick $\epsilon$ appropriately small as compared the constant of proportionality in the $\gg$ above) of these polynomials will be Reduced-$m$-Polynomials and hence will correspond to distinct rings of integers and thus number fields. We thus get that the number of number fields with $\disc<X$ whose ring of integers is weakly divisible by $m$ is 
$$
\gg X^{\frac{1}{2}+\frac{1}{n-1}}\frac{1}{m^{1-\frac{2}{n-1}+\frac{1}{n-2}-\epsilon}}=X^{\frac{1}{2}+\frac{1}{n-1}}\frac{1}{m^{1-\frac{n-3}{(n-1)(n-2)}+\epsilon}}.
$$

The various restrictions namely, 
$$
m^{2+\epsilon}\le st^n \textit{ and  } m^{2+\epsilon}\le \frac{(st^n)^{\eta_n}}{t^{\frac{n(n+1)}{2}+1}}
$$ 
for the chosen $t,$ combine into 
$$
m^{1+\epsilon} \le X^{r_n},
$$
where 
$$
r_n=\frac{\eta_n}{n^2-4n+3-\eta_n (n+\frac{2}{n-2})}.
$$

Summing over $m$ we get that the number of number fields with discriminant $\le X$ is 
$$
\gg X^{\frac{1}{2} +\frac{1}{n-1} +  \frac{(n-3)r_n}{(n-1)(n-2)}-\epsilon}
$$
\begin{remark}
    To get $S_n$ number fields we again can simply change the local ratio appropriately for a few (maximum 3) fixed small primes. See \cref{Matt-Baker-Argument}.
\end{remark}
\begin{remark}
    It is easy to see that a careful choice of $\epsilon$'s along the way and careful reorganization of the argument gets rid of the $-\epsilon$ in the power of $X$ in the lower bound except for the $\theta_M$ being $O(\frac{1}{\log(m)})^2$. This however can also be removed if we refresh the condition $p|m \Rightarrow p\nmid \frac{\disc(f)}{m^2}$ in \cref{local-cond} to $p|m \Rightarrow p^2\nmid \frac{\disc(f)}{m^2}$ which will also work in the argument and make $\theta_M= O(1).$ We will elaborate on this in papers in preparation.
\end{remark}


\begin{remark}    
    In \cite{VAALER_WIDMER_2015}, the authors show that if for some $c>0,$ the number of number fields with discriminant $<X$ is 
    $$
    \gg X^{\frac{1}{2}+\frac{1}{n-1}+c},
    $$
    then there exist infinitely many number-fields with no small generators.
    
    Our result, thus completes the answer to the famous Rupert's 1998 question about the relationship between Mahler measure of polynomials and discriminant of associated number fields.
\end{remark}
\section*{Acknowledgements}
This work is based of Chapter 1 of the author's thesis in \cite{Gaurav}. The author would like to thank Prof. V.K. Murty for the many helpful discussions and guidance. The author would like to thank Dr. H. Senapati for editorial help and Prof. S. Pathak for help with Latex.
\appendix
\section{\bf Matrix of transformation of canonical basis:$R_f\longrightarrow R_{f_l}$}

\bigskip
\begin{equation}
\begin{split}
    & \text{ Let } F(x,y)=\sum_{i=0}^n a_ix^{n-i}y^i \text{ denote an irreducible integral binary form.}\\
    &\text{Let } \delta \text{ denote the root of } f(x):=F(x,1).\\
    &\text{Let } F_l(x,y):=f(x+ly,y)=\sum_{i=0}^n b_ix^{n-i}y^i.\\
    &\text{Note that } \delta-l \text{ is the root of } f_l(x):=F_l(x,1).\\
    &\text{Let } s_i(x):=\sum_{j=0}^{i-1}a_jx^{i-j}=a_0x^{i}+a_1x^{i-1}+...+a_{i-1}x.\\
    &\text{Let }m_i(x):=\sum_{j=0}^{i-1}b_jx^{i-j}=b_0x^{i}+b_1x^{i-1}+...+b_{i-1}x.
\end{split}
\end{equation}
Thus the basis for $R_{F_l}$ will be given by \begin{align*}
    \Z\ncp{1,\: m_1(\delta-l) +b_1, \: m_2(\delta-l)+b_2,\:\cdots, \:m_k(\delta-l)+b_k,\:\cdots,\:m_{n-1}(\delta-l)+b_{n-1}}
\end{align*}
and that of 
$R_{f}$ will be given by \begin{align*}
    \Z\ncp{1,\: s_1(\delta) +a_1, \: s_2(\delta)+a_2,\:\cdots, \:s_k(\delta)+a_k,\:\cdots,\:s_{n-1}(\delta) +a_{n-1}}
\end{align*}
\begin{lemma}
 \begin{equation}
     b_k = \sum_{i=0}^{k}   {{n-i}\choose{n-k}}a_{i} \\
 \end{equation}
\end{lemma}
\begin{proof}
$$f(x+l)=\sum_{i=0}^n a_i(x+l)^{n-i}=\sum_{i=0}^n a_i\sum_{j=0}^{n-i}{n-i\choose j}x^jl^{n-i-j}$$
$$= \sum_{i=0}^n\sum_{j=0}^{n-i}a_i{n-i\choose j}x^jl^{n-i-j}=\sum_{j=0}^{n}x^{j}\sum_{i=0}^{n-j}a_i{n-i\choose j}l^{n-i-j}$$
Thus,
$$b_k=\sum_{i=0}^{k}a_i{n-i\choose n-k}l^{k-i}$$
\end{proof}
\begin{lemma}\label{2.9}
\begin{align*}
    m_k(x)+b_k&=x(m_{k-1}(x)+b_{k-1})+b_k\\
    s_k(x)+a_k&=x(s_{k-1}(x)+a_{k-1})+a_k
\end{align*}
\begin{proof}
By definition.
\end{proof}
\end{lemma}
\begin{proposition}\label{+1}
\begin{equation}
     m_k(x-l) + b_k =\sum_{j=0}^{k-1}{n-k+j-1\choose j}(s_{k-j}(x)+a_{k-j})l^{j} +  {n-1\choose k} l^{k}a_0\\
 \end{equation}
\end{proposition}
\begin{proof}
We proceed by induction on $k$,\\
Base Case.\\
$m_1(x-l)=b_0(x-l)+b_1=(a_0)(x-l) +a_1+na_0 l$\\
$=(a_0x+a_1)+{n-1\choose 1}a_0 l$\\
$={n-2\choose 0}(s_1(x)+a_1)+{n-1\choose 1}a_0$\\\\
Induction Hypothesis: $$ m_k(x-l) + b_k =\sum_{j=0}^{k-1}{n-k+j-1\choose j}(s_{k-j}(x)+a_{k-j})l^j +  {n-1\choose k} l^{k}a_0$$\\
We have from \cref{2.9},
$$m_{k+1}(x-l) = (x-l)(m_k(x-l)+ b_k).$$
Substituting in our induction hypothesis, we get
{\allowdisplaybreaks
\begin{align*}
    &m_{k+1}(x-l)\\
    &=(x-l)(\sum_{j=0}^{k-1}{n-k+j-1\choose j}(s_{k-j}(x)+a_{k-j})l^j +  {n-1\choose k} l^ka_0)&\\
    &=\sum_{j=0}^{k-1}{n-k+j-1\choose j}((s_{k-j}(x)+a_{k-j})l^jx-(s_{k-j}(x)+a_{k-j})l^{j+1}) +  {n-1\choose k} a_0l^k x -{n-1\choose k}l^{k+1}a_0\\
    &=(\sum_{j=0}^{k-1}{n-k+j-1\choose j}(s_{k-j}(x)x+a_{k-j}x)l^j +{n-1\choose k} a_0xl^k)\\&\text{\tab\tab\tab\tab\tab\tab}-(\sum_{j=0}^{k-1}{n-k+j-1\choose j}(s_{k-j}(x)+a_{k-j})l^{j+1} + {n-1\choose k}a_0l^{k+1})\\
    &=(\sum_{j=0}^{k-1}{n-k+j-1\choose j}(s_{k-j+1}(x))l^j +{n-1\choose k} s_1(x)l^k) \\&\text{\tab\tab\tab\tab\tab\tab}-(\sum_{j=0}^{k-1}{n-k+j-1\choose j}(s_{k-j}(x)+a_{k-j})l^{j+1} + {n-1\choose k}a_0l^{k+1})\\
    &=(\sum_{j=0}^{k}{n-k+j-1\choose j}(s_{k-j+1}(x))l^j -(\sum_{j=0}^{k-1}{n-k+j-1\choose j}(s_{k-j}(x)+a_{k-j})l^{j+1} + {n-1\choose k}a_0l^{k+1})\\
    &=(\sum_{j=0}^{k}{n-k+j-1\choose j}(s_{k-j+1}(x)+a_{k-j+1})l^j-(\sum_{j=0}^{k-1}{n-k+j-1\choose j}(s_{k-j}(x)+a_{k-j})l^{j+1})\\ &\text{\tab\tab\tab\tab\tab\tab\tab\tab}- {n-1\choose k}a_0l^{k+1} -\sum_{j=0}^{k-1}{n-k+j-1\choose j}a_{k-j+1}l^j\\
    &=\sum_{j=0}^{k}({n-k+j-1\choose j}-{n-k+j-2\choose j-1})(s_{k+1-j}+a_{k+1-j})l^j)\\&\text{\tab\tab\tab\tab\tab\tab}-{n-1\choose k}a_0l^{k+1}-\sum_{j=0}^{k-1}{n-k+j-1\choose j}a_{k-j+1}l^j\\
    &=(\sum_{j=0}^{k}{n-k-2+j\choose j}(s_{k+1-j}+a_{k+1-j})l^j)-{n-1\choose k}a_0l^{k+1}-\sum_{j=0}^{k}{n-k+j-1\choose j}a_{k-j+1}l^j\\
    &=(\sum_{j=0}^{k}{n-k-2+j\choose j}(s_{k+1-j}+a_{k+1-j})l^j +{n-1\choose k+1} a_0l^{k+1})\\&\text{\tab\tab\tab\tab\tab}-{n-1\choose k}a_0l^{k+1}-(\sum_{j=0}^{k}{n-k+j-1\choose j}a_{k-j+1}l^j+{n-1\choose k+1} a_0l^{k+1})\\
    &=(\sum_{j=0}^{k}{n-k-2+j\choose j}(s_{k+1-j}+a_{k+1-j})l^j +{n-1\choose k+1} a_0l^{k+1})\\&\text{\tab\tab\tab\tab\tab}-(\sum_{j=0}^{k}{n-k+j-1\choose j}a_{k-j+1}l^j+({n-1\choose k+1}+{n-1\choose k}) a_0l^{k+1})\\
    &=(\sum_{j=0}^{k}{n-k-2+j\choose j}(s_{k+1-j}+a_{k+1-j})l^j +{n-1\choose k+1} a_0l^{k+1})\\&\text{\tab\tab\tab\tab\tab}-(\sum_{j=0}^{k}{n-k+j-1\choose j}a_{k-j+1}l^j+{n\choose k+1} a_0l^{k+1})\\
    &=(\sum_{j=0}^{k}{n-k-2+j\choose j}(s_{k+1-j}+a_{k+1-j})l^j +{n-1\choose k+1} a_0l^{k+1})\\&\text{\tab\tab\tab\tab\tab}-(\sum_{j=0}^{k}{n-k+j-1\choose n-k-1}a_{k-j+1}l^j+{n\choose k+1} a_0l^{k+1})\\
    &=(\sum_{j=0}^{k}{n-k-2+j\choose j}(s_{k+1-j}+a_{k+1-j})l^j +{n-1\choose k+1} a_0l^{k+1})\\&\text{\tab\tab\tab\tab\tab}-b_{k+1}
\end{align*}
}
\end{proof}
Substituting $x=\delta$ in the above Proposition we get the following.

If we let $f(X,Y)=a_0 X^n+a_1X^{n-1}Y+\cdots+a_nY^n$ and  $\ncp{B_0,B_1,\cdots,B_{n-1}}$ be the canonical basis of $R_f$ associated to $f(X,Y)$ and  $\ncp{C_0,C_1,\cdots,C_{n-1}}$ be the canonical basis of $R_{f_l}$ associated to $f_l,$ then we have
\begin{align*}
    \begin{matrix}
        \ncp{C_0,C_1,\cdots,C_{n-1}}= \\ \\ \\ \\ \\ \\ \\ \\ \\  
    \end{matrix}\begin{matrix}
        \ncp{B_0,B_1,\cdots,B_{n-1}}\\ \\ \\ \\ \\ \\ \\ \\ \\ 
    \end{matrix}\begin{bmatrix}
    1 & \binom{n-1}{1} l a_0 & \binom{n-1}{2} l^2 a_0 &\cdots & \binom{n-1}{k-1}l^{k-1}a_0&\cdots &\binom{n-1}{n-1}l^{n-1}a_0\\
    0 & 1 & \binom{n-2}{1}l & \cdots &\binom{n-2}{k-2}l^2&\cdots &\binom{n-2}{n-2}l^{n-2}\\ 
    0&0&1&\cdots &\binom{n-3}{k-3}l^{k-3}&\cdots&\binom{n-3}{n-3}l^{n-3}\\
    \vdots&\vdots &\vdots& \ddots&\vdots&\ddots &\vdots \\
    0 &0&0&\cdots& \binom{n-k'}{k-k'}l^{k-k'}&\cdots & \binom{n-k'}{n-k'}l^{n-k'}\\
    \vdots&\vdots &\vdots& \ddots&\vdots&\ddots &\vdots \\
    0 & 0 &0 &\cdots & \binom{1}{k-n+1}l^{k-n+1}&\cdots & l\\
    0 & 0 &0 &\cdots & \binom{0}{k-n}l^{k-n}&\cdots & 1
\end{bmatrix}
\end{align*}
\bibliographystyle{abbrv}
\bibliography{main}
\end{document}